\documentclass[12pt,a4paper]{amsart}
\usepackage[margin=2.9cm]{geometry}
\usepackage{amsfonts}
\usepackage{amsthm}
\usepackage{amsgen}
\usepackage{amssymb}
\usepackage{mathtools}
\usepackage{amsmath}

\usepackage{scalerel}
\usepackage{amscd}
\usepackage{comment}
\usepackage[latin2]{inputenc}
\usepackage{t1enc}
\usepackage[mathscr]{eucal}
\usepackage{indentfirst}
\usepackage{graphicx}
\usepackage{graphics}
\usepackage{anyfontsize}
\usepackage{pict2e}
\usepackage{epic}
\numberwithin{equation}{section}
\usepackage[margin=2.9cm]{geometry}
\usepackage{epstopdf} 
\usepackage{cite}
\usepackage{tikz-cd}
\usetikzlibrary{arrows ,matrix, scopes, decorations.pathmorphing}

\usepackage{chngcntr}
\counterwithout{equation}{section}

\usepackage{amsthm}
\theoremstyle{plain} 
\newtheorem{thm}{Theorem}[section]
\newtheorem*{thm*}{Theorem}
\newtheorem*{cjt}{Conjecture} 
\newtheorem{cor}[thm]{Corollary}
\newtheorem{prop}[thm]{Proposition}
\newtheorem*{prop*}{Proposition}

\theoremstyle{definition}

\newtheorem{defn}{Definition}
\newtheorem*{ackn}{Acknowledgement}

\theoremstyle{remark}

\begin{document}
\title{Primitive characters of odd order groups}

\author[C. Marchi]{Claudio Marchi}

\address{School of Mathematics \\ University of Manchester \\
Oxford Road \\  M13 9PL \\ Manchester \\
United Kingdom} 

\email{claudio.marchi@manchester.ac.uk}



\begin{abstract} Let $G$ be a finite group of odd order. We show that if $\chi$ is an irreducible primitive character of $G$ then for all primes $p$ dividing the order of $G$ there is a conjugacy class such that the $p-$part of $\chi(1)$ divides the size of that conjugacy class. We also show that for some classes of groups the entire degree of an irreducible primitive character $\chi$ divides the size of a conjugacy class.
\end{abstract}

\maketitle

\section{Introduction}
Irreducible characters and conjugacy classes are related by many arithmetical properties. The most well-known is an immediate consequence of the Ito-Michler theorem saying that, if a prime number $p$ divides the degree of an irreducible character then $p$ divides the size of at least a conjugacy class. This result has been extended by Casolo and Dolfi in \cite{casolo2009products}, to show that if $pq$ divides the degree of an irreducible character then there exists a conjugacy class whose size is divisible by $pq$.\\
\indent In \cite{isaacs2006fixed} Isaacs, Keller, Meierfrankenfeld and Moreto' conjectured the following:
\begin{cjt}[\cite{isaacs2006fixed},Conjecture C]
Let $\chi$ be an irreducible primitive character of a finite group $G$. Then $\chi(1)$ divides $\vert cl_G(g)\vert$ for some element $g\in G$.
\end{cjt}
In \cite{cassell2013conjugacy} Cassell proved that the conjecture holds for all finite simple groups and symmetric groups. One way of approaching the conjecture is to consider only the $p-$parts of the character degrees and in \cite{isaacs2006fixed} the following was proved:
\begin{thm*}[\cite{isaacs2006fixed},Corollary D]
Let $\chi$ be an irreducible primitive character of a solvable group $G$ and let $p$ be a prime divisor of $\vert G\vert$. Then $\chi(1)_p$ divides $(\vert cl_G(g)\vert_p)^3$ for some element $g\in G$.
\end{thm*}
By restricting our attention to odd order groups we are able to prove a result closer to the original conjecture:
\begin{thm*}
Let $\chi$ be an irreducible primitive character of an odd order group $G$, $p$ a prime dividing $\vert G\vert$. Then there is a $g\in G$ such that $\chi(1)_p$ divides $\vert cl_G(g)\vert_p$.
\end{thm*}
In particular this result gives a positive answer to the conjecture for primitive irreducible characters of prime power order of odd order groups.\\
\indent We follow the same approach as in \cite{isaacs2006fixed} but focusing on odd order groups we can use the symplectic structure of some modules to get sharper bounds on the order of Sylow subgroups. A crucial result in proving the theorem above is in fact the following:
\begin{thm*}
Let $V$ be a symplectic, faithful $\mathbb{F}_3G$-module that is a direct sum of irreducible anisotropic modules and $dim(V)=2n$. If $\vert G\vert$ is odd, $O_3(G)=1$ and $P\in Syl_3(G)$ we have $j=log_3(\vert P\vert)\leq \beta_3\left(\frac{2}{3}n\right)$.
\end{thm*}
We remark that it is not clear how to use the result on $p-$parts for proving the conjecture on the entire degree of a primitive character, since the conjugacy class whose size is divided by the $p-$part depends on the chosen prime. However we managed to prove that the conjecture holds also for the following family of groups:
\begin{prop*}
Let $G$ be a meta-nilpotent group, $\vert G\vert$ odd, $\pi(G)=\left\lbrace p_1,p_2\right\rbrace$, $\chi\in$Irr(G) primitive character. Then there is $g\in G$ such that $\chi(1)$ divides $\vert cl_G(g)\vert$.
\end{prop*}

\section{Symplectic modules}\label{simplsect}
In this section we collect all the basic information on symplectic modules and we prove theorems that will be crucial for proving the main result. From now on $G$ will be a finite group, $p$ will be a prime number, $\mathbb{F}$ will be a field of positive characteristic, while $\mathbb{F}_p$ will be the field with $p$ elements. 
\begin{defn}
A vector space $V$ is symplectic if it carries an alternating bilinear form such that $V^{\perp}=0$.\\
\indent An $\mathbb{F}G-$module $V$ is a symplectic module if it a symplectic space such that the symplectic form is $G-$invariant, i.e. $\langle x^g,y^g\rangle =\langle x,y\rangle $ for all $x,y\in V$, $g\in G$.
\end{defn}
\begin{defn}
A submodule $U$ of a symplectic $\mathbb{F}G-$module $V$ is totally isotropic if $U\leq U^{\perp}$. \\
\indent $V$ is hyperbolic if it has a totally isotropic submodule with dimension $\frac{1}{2}dim(V)$.\\
\indent $V$ is anisotropic if it does not have non trivial totally isotropic submodules.
\end{defn}

\begin{prop}[\cite{wolf1984sylow}]
Let $\beta_3,\lambda_3$ be the following functions:
\begin{align*}
\beta_3(x)=\sum_{i=0}^{\infty}\left\lfloor \frac{x}{2\cdot 3^i}\right\rfloor\;\;\;\;\lambda_3(x)=\sum_{i=1}^{\infty}\left\lfloor \frac{x}{3^i}\right\rfloor
\end{align*}
If $a,b,m,d,n$ are positive integers and $\eta(x)$ is one of the previously defined functions, then the following hold 
\begin{enumerate}
\item $\eta_3(a+b)\geq \eta_3(a)+\eta_3(b)$
\item $\eta_3(md)\geq d\eta_3(m)+\lambda_3(d)$
\item $\lambda_3(n)\leq (n-1)/2$
\item $\beta_3(n)\leq \frac{3}{4}n-1/2$
\item $\beta_3\left(\frac{2}{3}md\right)\geq d\beta_3\left( \frac{2}{3}m\right)+\lambda_3(d)$
\end{enumerate}
\end{prop}
The following theorem follows the same proof of \cite[3.4]{loukaki2003hyperbolic}.
\begin{thm}\label{SelfDual}
Let $V$ be a self dual,irreducible, faithful $\mathbb{F}G$-module and $\vert G\vert$ an odd order group. Let $N$ be a normal subgroup of $G$ such that
\[
V_N=e(V_1\oplus...\oplus V_t)
\]
with $V_i$ irreducible $\mathbb{F}N-$modules and $t> 1$. Then $e$ is odd.
\end{thm}
\begin{proof}
Take $\mathbb{E}$ splitting field for $G$ and $N$, and write the following:
\begin{align*}
V^{\mathbb{E}}=\bigoplus_{j=1}^sW^j\\
V_i^{\mathbb{E}}=\bigoplus_{j=1}^{d_i}V_i^j.
\end{align*}
We remark that the $V_i^j$s are distinct, and the $W^j$ are distinct as well.\\
\indent Let $\mathbb{E}_V$ be the subfield of $\mathbb{E}$ obtaining adjoining to $\mathbb{F}$ the values of the characters corresponding to the $W_j$s. Actually, we just need to consider the values of the character corresponding to $W_1$ since the $W_j$s are Galois conjugates.\\
\indent We define $\mathbb{E}_{V_i}$ analogously. Again since the $V_i$ are $G-$conjugates, then the irreducible constituents of $V_k^{\mathbb{E}}$ are conjugate of $V_1^{\mathbb{E}}$, so the fields $\mathbb{E}_{V_i}$ are all the same field. Thus, since $d_i=\vert\mathbb{E}_{V_i}:\mathbb{F}\vert$, we have $d_i=:d$ for all $i$.\\
\indent Now from \cite[1.16]{huppert1998character} we have $W^i_{\mathbb{F}}\simeq \vert \mathbb{E}:\mathbb{E}_V\vert V$ and $(V_i^j)_{\mathbb{F}}\simeq \vert \mathbb{E}:\mathbb{E}_{V_1}\vert V_i$. Suppose then that $W^1,...,W^c$ are the constituents of $V^{\mathbb{E}}$ that have $V_1^1$ as a constituent after taking the restriction to $N$. Then
\begin{equation}\label{eqz}
W^1_N=...=W^c_N=e'(V_1^1\oplus...\oplus V^{k_{t'}}_{l_{t'}}),
\end{equation}
where the $G-$conjugate of $V_1^1$ are $V_{l_r}^{k_r}$, where $l_1=1=k_1$. But since we are talking of absolutely irreducible modules $e'$ divides $\vert G\vert$, and in particular $t'$ divides $\vert G\vert$. If now we take the restriction to $\mathbb{F}$ in the previous equality we get $(W^1_{\mathbb{F}})_N=e'((V_1^1)_{\mathbb{F}}\oplus...\oplus (V^{k_{t'}}_{l_{t'}})_{\mathbb{F}}).$ But then $\vert\mathbb{E}:\mathbb{E}_V\vert V_N=e'\vert \mathbb{E}:\mathbb{E}_{V_1}\vert(V_1\oplus V_{l_2}...\oplus V_{l_{t'}})$. Thus 
\[
et\vert \mathbb{E}:\mathbb{E}_V\vert=e't'\vert\mathbb{E}:\mathbb{E}_{V_1}\vert.
\]
Take $\mathbb{D}$ the subfield of $\mathbb{E}$ generated by $\mathbb{E}_V$ and $\mathbb{E}_{V_1}$. So dividing by $\vert\mathbb{E}:\mathbb{D}\vert$ we have $et\vert \mathbb{D}:E_V\vert=e't'\vert\mathbb{D}:\mathbb{E}_{V_1}\vert$.\\
\indent Suppose by contradiction that $e$ is even, then $\vert \mathbb{D}:\mathbb{E}_{V_1}\vert$ must be even as well. Let $\Gamma=Gal(\mathbb{D}\vert\mathbb{F})$. $\Gamma$ is cyclic, so it contains just an involution $\pi$. Write $\mathbb{E}_V^*=Fix_{\Gamma}(\mathbb{E}_V)$ e $\mathbb{E}_{V_1}^*=Fix_{\Gamma}(\mathbb{E}_{V_1})$.\\
\indent Now $\vert \mathbb{D}:\mathbb{E}_{V_1}\vert=\vert \mathbb{E}_{V_1}^*:1\vert$ is even and so $\mathbb{E}_{V_1}^*$ contains $\pi$, i.e. $\pi$ fixes point-wise $\mathbb{E}_{V_i}$, and so it fixes all the $V_i^j$. But $\pi$ acts non trivially on $\mathbb{D}$ and so on $\mathbb{E}_V$. Thus $\pi$ as a $\mathbb{F}$ automorphism of $\mathbb{E}_V$ is the unique involution of $Gal(\mathbb{E}_V\vert \mathbb{F})$. Now $V$ is self dual, so $V^{\mathbb{E}}$ is self dual as well. If $W^1$ were self dual we take $\psi$ the corresponding Brauer character and by \cite[4.14.14]{karpilovsky2016group} we can find an irreducible ordinary character $\chi$ such that $d_{\chi\psi}$ is odd. But now $G$ is an odd order group, so there are no non trivial self dual irreducible characters, thus $\chi=1_G$ and $\psi$ is the trivial Brauer character. So $W^1$ is trivial and there is just one $W^i$, i.e. $V^{\mathbb{E}}$ is trivial and $V$ is trivial, but this can not happen. So all the $W^j$s are not self dual and $\pi$ must send every $W^i$ to its dual. \\ 
\indent Now since $\pi$ fixes the $V_i^j$, applying $\pi$ to the equation (\ref{eqz}) yields
\[
\begin{split}
e'(V_1^1\oplus...\oplus V_{l_{t'}}^{k_{t'}})&\simeq e'(\pi(V_1^1)\oplus...\oplus \pi(V_{l_{t'}}^{k_{t'}}))\simeq\pi(W_N^1)\\
&\simeq\hat{W}^1_N\simeq e'(\hat{V}_1^1\oplus...\oplus \hat{V}_{l_{t'}}^{k_{t'}})
\end{split}
\]
If $V_1^1$ is self dual then as before $V_i$ is trivial and we have a contradiction. Thus no $V_{l_r}^{k_r}$ is self dual.\\
\indent We can then match each $V_{l_r}^{k_r}$ with its dual and then $t$ would be even, but $t$ must divide $\vert G\vert$ and we have the final contradiction.
\end{proof}
\begin{prop}\label{hyp}
If $V$ is a symplectic, homogeneous, completely reducible module, $V=eW$, $e> 1$, with all irreducible submodules totally isotropic and self dual, then $V$ admits a decomposition as orthogonal direct sum of symplectic hyperbolic modules, with dimension $2dim(W)$
\[
V=H_1\perp...\perp H_{t}.
\]
$\left[ \textnormal{Remark: }e\textnormal{ is odd since }dim(V)=e\cdot dim(W)=2t\cdot dim(W)\right]$
\end{prop}
\begin{proof}We proceed by induction on the composition length of $V$. \\
\indent If $V$ has length $2$ then it is trivial.\\
\indent Suppose that the composition length is greater than $2$. Let $W$ be an irreducible submodule of $V$, and $N$ a $V-$complement of $W^{\perp}$. By \cite[10.22]{isaacs2018characters}, $N$ is isomorphic to the dual of $W$, so it is isomorphic to $W$ and totally isotropic . The sum $W+N$ is direct, and the restriction of the form is non-singular.\\
\indent In fact if $w\in W$, $n\in N$ and $\langle w+n,u\rangle =0$ for all $u\in W+ N$, in particular it holds $\langle n,u\rangle =0$ for all $u\in W$, i.e $n\in W^{\perp}\bigcap N=0$.\\
So $(W+ N)\bigcap (W+ N)^{\perp}\leq W$. If this intersection is $W$ then in particular $W\leq N^{\perp}$, i.e $N\leq W^{\perp}$ that is a contradiction.\\
\indent So $H_1=W\oplus N$ is an hyperbolic symplectic module. Now
\[
\begin{split}
H_1\bigcap H_1^{\perp}&=\left(W\oplus N\right)\bigcap (W\oplus N)^{\perp}=(W\oplus N)\bigcap W^{\perp}\bigcap N^{\perp}\\
&=\left(\left(W\oplus N\right)\bigcap W^{\perp}\right)\bigcap N^{\perp}=W\bigcap N^{\perp}=0.
\end{split}
\]
Thus $V=H_1\perp H_1^{\perp}$. Now $H_1^{\perp}$ is symplectic, since $(H_1^{\perp})^{\perp}=H_1$. In particular it is non irreducible and so it has composition length greater than $1$ and we can use the inductive hypothesis.
\end{proof}
\begin{thm}\label{sylow}
Let $V$ be a symplectic, faithful $\mathbb{F}_3G$-module that is a direct sum of irreducible anisotropic modules and $dim(V)=2n$. If $\vert G\vert$ is odd, $O_3(G)=1$ and $P\in Syl_3(G)$ we have $j=log_3(\vert P\vert)\leq \beta_3\left(\frac{2}{3}n\right)$.
\end{thm}
\begin{proof}
We proceed by induction on $n$, let $V$ be a completely reducible, faithful, symplectic, $\mathbb{F}_3[G]$ module, with dimension $2n$. 
\\
\textbf{Step 1}:\textit{We can take $V$ irreducible}
\\
\indent Let $V=W_1\oplus ...\oplus W_d$ with $W_i$ irreducible anisotropic modules, with $2m_i=dim(W_i)$. So if $C_i=C_G(W_i)$, the $W_i$s are irreducible, faithful, symplectic $G/C_i$ modules, so by inductive hypothesis we have 
\[
log_3\left(\vert G/C_i\vert_3\right)\leq \beta_3\left(\frac{2}{3}m_i\right).
\]
Now $\bigcap C_i=1$, so $G$ is isomorphic to a subgroup of $G/C_1\times ...\times G/C_d$ and thus
\[
log_3\left(\vert G\vert_3\right)\leq\sum\beta_3\left(\frac{2}{3}m_i\right)\leq\beta_3\left(\frac{2}{3}\sum m_i\right)\leq\beta_3\left(\frac{2}{3}n\right)
\]
\\
\noindent
\textbf{Step 2}:\textit{ Suppose that $V$ is imprimitive}
\\
\indent Let $N\unlhd G$ maximal such that $V_N$ is not homogeneous, 
\[
V_N=e(V_1\oplus...\oplus V_t)=W_1\oplus...\oplus W_d,
\]
where the $W_i$s are the homogeneous components. Then $S=G/N$ permutes faithfully and primitively the homogeneous components of $V_N$. Now the homogeneous components are symplectic with respect to the induced form, and are all orthogonal to each other, by \cite[10.23]{isaacs2018characters} and \cite[2.5]{loukaki2003hyperbolic}.\\
\indent Now suppose that an homogeneous component contains an irreducible anisotropic submodule. Then since the homogeneous components are conjugated each of them contains an irreducible anisotropic submodule. Thus by using \cite[Lemma 4]{pacifici2005number} on each homogeneous component we have that $V$ is direct sum of simple non singular submodules. We can then apply the inductive hypothesis on each $W_i$, and if $N_i=N/C_N(W_i)$ we have
\[
log_3\left(\vert N_i\vert_3\right)\leq \beta_3\left(\frac{2}{3}m\right).
\]
Now doing as in step $1$ we get
\[
log_3\left(\vert N\vert_3\right)\leq d\beta_3\left(\frac{2}{3} m\right).
\]
Since $S$ is isomorphic to a subgroup of $S_d$, 
\[
log_3\left(\vert G\vert_3\right)\leq d\beta_3\left(\frac{2}{3} m\right)+\lambda_3(d)\leq \beta_3\left(\frac{2}{3} md\right)\leq \beta_3\left(\frac{2}{3} n\right).
\]
\indent If no homogeneous component contains an irreducible anisotropic submodule then by \ref{hyp} we have that $e$ is even but by \ref{SelfDual} $e$ must be odd, so we have a contradiction.
\\
\textbf{Step 3}: \textit{Suppose finally that $V$ is primitive}
\\
\indent If $V$ is primitive then every normal abelian subgroup of $G$ is cyclic, $F(G)=EU$ with $U=Z(F(G))$ cyclic, all the Sylow subgroups of $E$ are extraspecial or have prime order and $E\bigcap U=Z=Soc(U)$.
\\
\indent If $E=Z$, then the Fitting subgroup of $G$ is cyclic and since $G$ is solvable we know that $C_G(F(G))\leq F(G)$, i.e $C_G(F(G))=F(G)$. Now $V_{F(G)}$ is homogeneous, in particular simple by \cite[2.2]{manz1993representations}.\\
\indent So by \cite[2.1]{turull1983supersolvable} $V$ can be identified with $\mathbb{F}_{3^{2n}}^{+}$ and $G$ with a subgroup of the semidirect product between $\mathbb{F}_{3^{2n}}^{\times}$ and $\textit{Gal}\left(\mathbb{F}_{3^{2n}}\vert \mathbb{F}_{3}\right)$, with $G\bigcap \mathbb{F}_{3^{2n}}^{\times}=F(G)$. But then $\vert G/F(G)\vert$ is cyclic and its order divides $2n$. However $\vert G\vert$ is odd, so $\vert G/F(G)\vert\leq n=dim(V)/2$, thus
\[
log_3(\vert G/F(G)\vert_3)\leq\lambda_3(n)=\beta_3\left(\frac{2}{3}n\right).
\]
\\
\indent We can then suppose $E> Z$. Then $E/Z=Q_1/Z\times ...\times Q_k/Z$, and $e^2=\vert E:Z\vert=q_1^{2m_1}...q_k^{2m_k}$. Now if $A=C_G(Z)$ then $A/C_A(Q_i/Z)\lesssim Sp(2m_i,q)$ and if we consider $C_i=C_A(Q_i/Z)$, then by \cite[1.3]{wolf1984sylow} we have
\[
log_3(\vert A:C_i\vert_3)\leq \beta_3\left(q_i^{m_i}\right).
\]
Now we have that $C=\bigcap C_i=F(G)$.\\
\indent In fact $E\leq C$ and $U\leq C$ imply $F(G)\leq C$. $U$ is the centralizer of $E$ in $G$, thus $C/U$ acts faithfully on $E$, centralizing both $E/Z$ and $Z$. But $E/Z$ is abelian, so $\vert C/U\vert $ divides $\vert E:Z\vert$ by \cite[2.1]{wolf1982solvable}, thus $C/U=E/Z$, i.e. $C=F(G)$.\\
\indent Now by \cite[1.3]{wolf1984sylow}
\[
log_3(\vert A:F(G)\vert_3)\leq \sum \log_3(\vert A:C_i\vert_3)\leq \beta_3\left(\sum q_i^{m_i}\right).
\]
On the other hand $\vert G:A\vert\leq(\vert Z\vert-1)/2$, since $G/A$ embeds in $Aut(Z)$. Let then $s$ be the maximum integer such that $3^s\leq (\vert Z\vert-1)/2$, so we can write:
\[
log_3(\vert G:F(G)\vert_3)\leq\beta_3\left(\sum q_i^{m_i}\right)+s.
\]
Now from \cite[2.2]{yang2011regular} we have that $\vert V\vert = \vert W\vert^{eb}$, where $W$ is an irreducible $U-$module such that the action of $U$ is fixed point free. Obviously $\vert W\vert\geq \vert U\vert-1$. \\
\indent If $b> 1$ then $\vert V\vert\geq (\vert U\vert-1)^{2e}$. Otherwise $V_{U}$ is sum of $e$ irreducible submodules all isomorphic to $W$.\\
If there is an irreducible non singular submodule $M\leq V_{U}$ then $M\simeq W$ and thus $U$ embeds in $Sp(2m,3)$, with $2m=dim(W)$. So by \cite[1.1]{espuelas1991regular}
\[
\vert V\vert\geq \vert W\vert\geq\left( \vert U\vert-1\right)^{2e}.
\]
If every irreducible submodule of $V_{U}=eW$ is totally isotropic then we have a contradiction. In fact by \ref{hyp} $e$ is even, but $e\mid\vert G\vert$ so it would be odd.
\\
\indent We can then say that $\vert V\vert\geq (\vert U\vert-1)^{2e}$, and let $t$ be the minimum integer such that $3^t\geq (\vert U\vert-1)^2$, so $2n\geq te$. We remark that $\vert U\vert\geq 5$, so $t\geq 3$.\\
\indent If $e$ is not a prime power then since $a+b+c\leq abc/6$ for $a\geq 5,b\geq 7,c\geq 3$, we have 
\[
\begin{split}
log_3(\vert G:F(G)\vert_3)&\leq\beta_3\left(\sum q_i^{m_i}\right)+s\leq\beta_3\left(\sum q_i^{m_i}\right)+t/2\leq \beta_3\left(\sum q_i^{m_i}\right)+\beta_3\left( t\right)\\
                          &\leq \beta_3\left(\sum q_i^{m_i}+t\right)\leq \beta_3\left(et/6\right)\leq\beta_3\left( \frac{2}{3}n\right).
\end{split}
\]
Otherwise if $e=q^m$ is a prime power we observe that if $q^m+t\geq 13$ and $t,q^m\geq 5$ it holds $q^m+t\leq q^mt/3=et/3$, thus the theorem is proved in this case.\\
\indent Now we only need to check a few cases, i.e. the pairs $(t,q^m)$ that satisfy $t+q^m< 13$. In particular since $e\in \left\lbrace 5,7,11\right\rbrace,$ $t$ cannot be greater than $8$. Thus the cases we need to check are $(t\leq 7,5)$,$(t\leq 5,7)$, $(1,11)$. Moreover since $3^t\geq (\vert U\vert-1)^2$ we can exclude $(1,11)$.
\\
\indent Case $e=7, t\leq 5$. Since $7\vert e$ we have that $7$ is a prime divisor of $\vert U\vert$, so we can exclude the cases $t\leq 3$. If $t=5$, then $3^5\geq (\vert U\vert-1)^2\geq 3^4$, i.e $(\vert U\vert-1)/2\leq 8$, so $s=1$. Since $2n\geq et\geq 28$, then $\beta_3\left(\frac{2}{3}n\right)\geq \beta_3(28/3)=4+1$ and keeping in mind $\beta_3(7)=4$ it follows 
\[
\beta_3\left(\frac{2}{3}n\right)\geq 5\geq s+\beta_3(7)\geq log_3(\vert G:F(G)\vert_3).
\]
\\
\indent Case $e=5, t\leq 7$. Now $3^7\geq (\vert U\vert-1)^2$ thus $\vert U\vert$ can be equal to $5,25,35$. $Z$ is the socle of $U$ so $\vert Z\vert=35$ or $\vert Z\vert=5$. In particular $s=0$ or $s=2$, then $s\leq 2$. If $2n\geq 35$, i.e $t=7$, the following holds 
\[
\beta_3\left(\frac{2}{3}n\right)\geq 6\geq 4\geq 2+\beta_3(5)\geq s+\beta_3(e)\geq log_3(\vert G:F(G)\vert_3).
\]
If $t\leq 6$ then $(35-1)^2> 3^6$, so $\vert U\vert =5$ or $\vert U\vert=25$. In particular $s=0$, then
\[
\beta_3(2n/3)\geq 2=\beta_3(e)+0\geq log_3(\vert G:F(G)\vert_3).
\]
\end{proof}

\section{Main results}\label{sectchars}
In this section we finally deal with the degrees of primitive characters. We remind the definition of a primitive character.
\begin{defn}
We say that a character $\chi$ is primitive if it cannot be induced from any proper subgroup of $G$.
\end{defn}
For the benefit of the reader we state a couple of results that will be used.
\begin{thm}[\cite{wolf1984sylow},1.9]\label{Wolf}
Let $V$ be a faithful $\mathbb{F}_p[G]-$module. with $G$ $p-$solvable. Assume further that $O_p(G)=1$ and let $Q$ be a Sylow $p-$subgroup of $G$. Then $\vert Q\vert\leq \vert V\vert^{\alpha}/\sqrt[p-1]{p}$, where 
\begin{equation}
\alpha=
\begin{cases*}
p/(p-1)^2 & if $p$ is a Fermat prime \\
1/(p-1) & otherwise
\end{cases*}
\end{equation} 
\end{thm}

\begin{thm}[\cite{isaacs2006fixed},Theorem A]\label{Isaacs}
Let $V$ be a completely reducible $\mathbb{F}[G]$-module, where $\mathbb{F}$ is any field. Suppose that $C_V(G)=0$ and let $p$ be the smallest prime divisor of $\vert G\vert$. Then there exists $g\in G$ such that 
\begin{equation}
dimC_V(g)\leq \frac{1}{p} dimV.
\end{equation}
\end{thm}
We also state the definition of $p-$special characters. A comprehensive discussion of this subject can be found in \cite{gajendragadkar1979characteristic} or in the more recent \cite{isaacs2018characters}.
\begin{defn}
Let $\pi$ be a set of primes, $G$ a $\pi-$separable group. We say that a character $\chi\in Irr(G)$ is $\pi-$special if $\chi(1)$ is a $\pi-$number and $o(\theta)$ is a $\pi-$number for every irreducible constituent $\theta$ of the restriction $\chi_S$ for every subnormal subgroup $S\lhd\lhd G$. If $\pi=\left\lbrace p\right\rbrace$ we say that $\chi$ is a $p-$special character.
\end{defn}
\begin{thm}\label{MAIN}
Let $\chi$ be an irreducible primitive character of an odd order group $G$, $p$ a prime dividing $\vert G\vert$. Then there is a $g\in G$ such that $\left(\chi(1)\right)_p$ divides $\vert cl_G(g)\vert_p$.
\end{thm}
\begin{proof} Since, by \cite[2.9]{isaacs2018characters}, primitive characters are factorizable as a product of a (primitive)$p-$special character and a (primitive)$p'-$special character we can assume that $\chi$ is a $p-$special character. We can also assume that $\chi$ is faithful and so $O_{p'}(G)=1$ and $F(G)=O_p(G)$.\\
\indent Now since $G$ has a primitive faithful character, every normal abelian subgroup of $G$ is cyclic and central, thus there is a characterization of the Fitting subgroup: $F(G)=EZ(G)$ is the central product of $E$, an extraspecial $p-$group and the centre of the group. Moreover $V=E/Z(E)$ is a completely reducible, faithful, symplectic $\mathbb{F}_pG/F(G)$ module.\\
\indent Now if $P$ is a Sylow $p-$subgroup of $G$, the restriction map $\chi\mapsto \chi_P$ maps injectively the $p-$special characters of $G$ to the irreducible characters of $P$. Thus
\begin{equation}\label{nilp}
\chi(1)\mid\vert P:Z(P)\vert^{1/2}.
\end{equation}
Moreover $Z(P)=Z(F(G))=Z(G)$, in fact since $G$ is solvable $Z(P)\subseteq C_G(F(G))\subseteq Z(G)\subseteq O_p\subseteq P$.
\\
\indent Now by (\ref{nilp}), $\chi(1)$ divides 
\[
\vert P:F(G)\vert^{1/2}\vert F(G):Z(G)\vert^{1/2}=\vert P:F(G)\vert^{1/2} \vert E:Z\vert^{1/2}.
\]
We observe that $G$ acts fixed point freely on $V$.\\
\indent In fact if $xZ\in C_V(G)$, $x\in E$, then $x^gZ=xZ$ for all $g\in G$. Thus $\langle x\rangle Z$ is a normal subgroup of $G$. In particular $\langle x\rangle Z$ is also abelian, thus cyclic and central, that implies $x\in Z$ and $xZ=Z$. \\
\indent Now if we use Theorem \ref{Wolf} with $G:=G/F(G)$, $V=E/Z(E)$ and $Q$ a $p-$Sylow subgroup of $G/F(G)$, then $\vert Q\vert=\vert P:F(G)\vert$, and for $p\geq 5$ $\vert P:F(G)\vert\leq \vert V\vert^{1/3}$. Thus $\chi(1)\leq \vert E:Z\vert^{1/6+1/2}=\vert E:Z\vert^{2/3}$. Since the minimum prime that divides $\vert G\vert$ is at least $3$, if $g$ whose existence is guaranteed by Theorem \ref{Isaacs}, we have $\vert cl_G(g)\vert_p\geq V^{2/3}\geq \chi(1)$.
\\
\indent We only need to prove the Theorem in the case $p=3$. Now the module $V$ is the direct sum of irreducible anisotropic modules, that inherit the symplectic form induced by the commutator.\\
In fact if an irreducible module $N/Z$ were totally isotropic then for all $a,e\in N$ we would have $\langle aZ,eZ\rangle =0$, i.e. $[a,e]=1$. Thus $N$ would be abelian, but then it would be central and thus $N\leq E\bigcap Z(G)=Z$.\\
\indent So we may apply Theorem \ref{SelfDual} to $G/F(G)$ and $V$, thus we have 
\[
\vert G:F(G)\vert_3\leq 3^{\beta_3\left( \frac{2}{3}n\right)}=3^{\beta_3(dim(V)/3)}\leq 3^{dim(V)/4}=\vert V\vert^{1/4}
\]
and then 
\[
\chi(1)\leq \vert V\vert^{7/12}\leq \vert V\vert^{2/3}\leq \vert cl_G(g)\vert_3
\]
\end{proof}
\begin{cor}\label{SylowAbNor}
If $\vert G\vert=p^aq^b$ is odd, and $G$ has a normal or abelian Sylow subgroup, then for all primitive characters $\chi\in Irr(G)$ there is $g\in G$ such that $\chi(1)$ divides $\vert g^G\vert$.
\end{cor}
\begin{proof}
Let $\chi=\alpha\beta$ be a factorization of $\chi$ as a product of respectively a $p-$special character and a $q-$special. If $G$ has a normal Sylow $p-$subgroup $P$, then $\vert G:P\vert=q^b$ and $\beta_P$ is irreducible, hence is linear and $\chi(1)$ is a prime power.\\
\indent If $G$ has an abelian Sylow $p-$subgroup then $p\nmid \chi(1)$, and then $\chi(1)$ is again a prime power.\\
\indent We can now apply Theorem \ref{MAIN}.
\end{proof}
The following is an easy corollary of Theorem \ref{Isaacs} in the case $V$ is a module of mixed characteristic.
\begin{prop}\label{CentralizerPFactor}
Let $V$ be  a finite $G-$module of mixed characteristic, $G$ odd order group, $V=\hat{P}_1\oplus\hat{P}_2$, with $\hat{P}_i$ non trivial $\mathbb{F}_{p_i}G-$modules and $\pi(\vert G\vert)=\left\lbrace p_1,p_2 \right\rbrace$. If $K\unlhd G$, $C_V(K)=0$ and $V$ is $K-$completely reducible, then every coset of $K$ contains a $t$ such that $dim(C_{\hat{P}_i}(t))\leq \frac{1}{p_{i}}dim(\hat{P}_i)$.
\end{prop}
\begin{proof}
We use induction on $dim(V)$. We can suppose that $G=K\langle t\rangle $ and that $K$ acts faithfully. In fact if $K$ does not act faithfully, then we may consider $K/C_K(V)\unlhd G/C_K(V)$.\\
\indent If $G/C_K(V)$ is a $p_i-$group, then $K/C_K(V)$ is nilpotent and thus acts trivially on $\hat{P}_i$, since $\hat{P}_i$ is $K-$completely reducible.
But $C_V(K)=0$ implies that $\hat{P}_i$ is trivial, and this is a contradiction. So it must be that $\pi(G/C_K(V))=\pi(G)$ and we are done by induction.\\
\indent Suppose then that $K$ acts faithfully and consider $N\leq K$ minimal normal subgroup of $G$. If $U=C_V(N)>0$, then $U< V$ and $U$ admits the action of $G$, in fact if $v\in U$ and $x\in N$, then
\[
v^{gx}=v^{x^{-1}gx}=v^{[x,g^{-1}]g}=v^{x'g}=v^g.
\]
So $V=U\oplus W$ with $[V,N]=W$ is a sum of $G-$modules. In fact if $W$ is the $N-$complement of $U$, $[V,N]\subseteq W$. But by the $N-$complete reducibility $[V,N]$ has a complement, that must be $N-$invariant and thus $W=[V,N]$.
\\
\indent We remark that $C_U(K)=0$. Now if, say, $\hat{P}_1\bigcap U=0$ then by Theorem \ref{Isaacs} we have that $\exists x\in Kt$ such that
\[
dim\left(C_{\hat{P}_i\bigcap U}(x)\right)\leq \frac{1}{p}dim(\hat{P}_i\bigcap U).
\] 
Suppose then now that $\hat{P}_i\cap U\neq 0$ for all the $i$s. If $K$ is a $p_j-$group, then $K$ is nilpotent and so by comprimality as before we would have $\hat{P}_j\bigcap U=0$, but this is a contradiction. \\
\indent Then we must have that $\pi(K)=\left\lbrace p_1,p_2\right\rbrace$. So by induction $\exists x\in Kt$ such that 
\[
dim\left(C_{\hat{P}_i\bigcap U}(x)\right)\leq \frac{1}{p}dim(\hat{P}_i\bigcap U).
\] 
Now $C_W(N)=0$ and we can suppose that $N$ is an elementary abelian $p_1-$group, but $W\bigcap \hat{P}_1$ is a faithful completely reducible $N/C_N(W\bigcap\hat{P}_1)-$module, and thus $C_N(W\bigcap \hat{P}_1)=N$, i.e. $W\bigcap\hat{P}_1=0$. Then by Theorem \ref{Isaacs} there is $y\in Nx\subseteq Kt$ such that 
\[
dim\left(C_{\hat{P}_i\bigcap W}(y)\right)\leq \frac{1}{p}dim(\hat{P}_i\bigcap W).
\] 
But $y=nx$, so $C_U(y)=C_U(x)$ and then we are done.
\\
\indent If $C_V(N)=0$ we can replace $K$ with $N$, and suppose that it is an elementary abelian $p_1-$group. But then $\hat{P}_1$ is a faithful completely reducible $K/C_K(\hat{P}_1)-$module, so $C_K(\hat{P}_1)=K$, i.e. $K$ acts trivially on $\hat{P}_1$ and thus $\hat{P}_1=0$ but this is a contradiction. 
\end{proof}
We want to remark that the Proposition \ref{CentralizerPFactor}, without any restrictions on the prime factors of the order of $G$, is false, as we can see from the following example. \\
\indent Take $G=Kt=\langle k\rangle \langle t\rangle $ isomorphic to the Klein group. Take then $p_1,p_2$ odd prime numbers and $S_1,T_1$ the linear $\overline{\mathbb{F}}_{p_1}G$, $\overline{\mathbb{F}}_{p_2}G-$modules, such that if $R_S$ is the representation associated to $S_1$, $R_1:G\rightarrow \overline{\mathbb{F}}_{p_1}^*$, then $Ker(R_S)=\langle kt\rangle $ and as well $Ker(R_T)=\langle kt\rangle $. Now take $S_2,T_2$ linear $\overline{\mathbb{F}}_{p_1}G$,$\overline{\mathbb{F}}_{p_2}G-$modules such that the kernel of the associated representation is $\langle t\rangle $. \\
\indent Consider then:
\begin{align*}
V_1=2S_1\oplus S_2\\
V_2=T_1\oplus 2S_2
\end{align*}
These two modules have dimension $3$ obviously. If we look at the fixed points of the coset $Kt$ then
\begin{align*}
dimC_{V_1}(kt)=2> \frac{3}{2}\\
dimC_{V_2}(t)=2> \frac{3}{2},
\end{align*} 
so there is no element $g$ in the coset $Kt$ such that $dimC_{V_i}(g)\leq \frac{1}{2}dim(V_i)$.\\
\indent The following is the only case where we were able to prove that the degree of a primitive character divides the size of a conjugacy class, assuming that the degree is not necessarily a prime power.

\begin{prop}
Let $G$ be a meta-nilpotent group, $\vert G\vert$ odd, $\pi(G)=\left\lbrace p_1,p_2\right\rbrace$, $\chi\in$Irr(G) primitive character. Then there is $g\in G$ such that $\chi(1)$ divides $\vert cl_G(g)\vert$.
\end{prop} 
\begin{proof}
We can take $\chi$ faithful, and so we have again the following structure of the Fitting subgroup, $F(G)=EZ(G)$. We then assume that $\chi(1)=p_1^ap_2^b$, with $a,b\geq 1$, otherwise the thesis immediately follows from Proposition \ref{MAIN}. \\
\indent Let $V=E/Z$, $V=\hat{P}_1\oplus\hat{P}_2$, with $\hat{P}_i$ $\mathbb{F}_{p_i}-$modules, even trivial. Let $\chi=\alpha_1\alpha_2$ be a factorization of $\chi$ as a product of $p_i-$specials characters. Notice that if $G_i/F(G)$ are the Sylow $p_i-$subgroups of $G/F(G)$ then by \cite[5.1]{isaacs1981primitive} we have that $(\alpha_i)_{G_{i+1}}$ is primitive and irreducible, where the indices are modulo $2$.\\
\indent But then if $O_{p_i}:=O_{p_i}(G)$ we have that $(\alpha_i)_{O_{p_i}}$ is irreducible, and then $\alpha_i(1)$ divides $\vert O_{p_i}:Z(O_{p_i})\vert^{1/2}$. That is $\chi(1)$ divides $\vert V\vert$.\\ Now if, say, $\hat{P}_1$ is trivial then $O_{p_1}$ is central and thus $\alpha_1$ is linear, but we have already excluded this occurrence. So we can take $\hat{P}_i$ non trivial modules, and then by Proposition (\ref{CentralizerPFactor}) there is an element $g\in G$ such that 
\[
\vert g^G\vert_{p_i}\geq \vert E:Z\vert_{p_i}^{2/3}\geq \vert O_{p_i}:Z(O_{p_i})\vert^{2/3}\geq \alpha_i(1),
\]
thus $\vert g^G\vert$ is divided by $\chi(1)$.
\end{proof}
\begin{ackn}
I would like to thank Silvio Dolfi, for many useful discussions and comments, and Charles Eaton, for a careful reading of the manuscript.
\end{ackn}
\bibliographystyle{siam}
\bibliography{Bibliography}
\end{document}